\documentclass[11pt]{article}
\usepackage{amsfonts,amsmath,latexsym}
\usepackage{amssymb}
\def\R{\mathbb R}
\def\N{\mathbb N}

\def\E{\mathbb E}

\def\shb{{\cal B}}
\def\shc{{\cal C}}
\def\shd{{\cal D}}
\def\shf{{\cal F}}
\def\shm{{\cal M}}

\def\shp{{\cal P}}
\def\shs{{\cal S}}
\def\shy{{\cal Y}}
\newtheorem{theo}{Theorem}[section]

\newtheorem{lemma}[theo]{Lemma}

\newtheorem{prop}[theo]{Proposition}
\newtheorem{rem}[theo]{Remark}
\newtheorem{cor}[theo]{Corollary}
\newtheorem{defi}[theo]{Definition}
\newtheorem{property}[theo]{Property}
\newtheorem{Assumption}[theo]{Assumption}

\newcommand{\beqnar}{\begin{eqnarray*}}
\newcommand{\eeqnar}{\end{eqnarray*}}
\newcommand{\ba}{\begin{array}}
\newcommand{\ea}{\end{array}}

\newcommand{\halb}{\frac{1}{2}}

\newenvironment{proof}[1]{\begin{trivlist}\item {\it
\bf Proof.}\quad} {\qed\end{trivlist}}
\newenvironment{prooff}[1]{\begin{trivlist}\item {\it
\bf Proof}\quad} {\qed\end{trivlist}}

\newcommand{\qed}{\nopagebreak\hspace*{\fill}
{\vrule width6pt height6ptdepth0pt}\par}

\begin{document}

\title{Probabilistic representation for solutions
of an irregular porous media type  equation.}

\author{ Philippe Blanchard (1), Michael R\"ockner (2)
and Francesco Russo (3) } 

\date{}
\maketitle

\thispagestyle{myheadings}
\markright{Irregular porous media type equation}

{\bf Summary:} We consider a
porous media type equation over all of $\R^d$ with $d = 1$,
 with monotone  discontinuous coefficient with linear growth and prove a
probabilistic representation of its solution in terms of an associated 
microscopic diffusion. 
The interest in such a singular porous media equations is due to the
fact that they can model systems exhibiting the phenomenon
of self-organized criticality.
One of the main analytic ingredients of the proof is a new result
on uniqueness of distributional solutions of a linear PDE on $\R^1$
 with non-continuous coefficients.


{\bf2000  AMS-classification}: 60H30, 60H10, 60G46,
 35C99, 58J65


{\bf Actual version:} October 10th 2009

\bigskip

\begin{itemize}
\item[(1)] Philippe Blanchard,
Fakult\"at f\"ur Physik, 
Universit\"at   Bielefeld,  D--33615 Bielefeld,
Germany.
\item[(2)] Michael R\"ockner,
Fakult\"at f\"ur Mathematik, 
Universit\"at   Bielefeld, 
\\ D--33615 Bielefeld, Germany  and \\
Department of Mathematics and Statistics, 
Purdue University, \\
W. Lafayette, IN 47907, USA.
\item[(3)] Francesco Russo,
INRIA Rocquencourt, Equipe MathFi and Cermics
Ecole des Ponts,
Domaine de Voluceau,
Rocquencourt - B.P. 105,
F-78153 Le Chesnay Cedex, France\\ and
Universit\'{e} Paris 13, Institut Galil\'{e}e, Math\'ematiques,
\end{itemize}

\vfill \eject

\section{Introduction}

We are interested in the probabilistic representation of the solution to
 a porous media type equation  given by
\begin{equation}
\label{E1.1}
\left \{
\begin{array}{ccc}
\partial_t u&=& \frac{1}{2} \partial_{xx}^2(\beta(u)),  \ t \in [0, \infty[
\\
u(0,x)& = & u_0(x), \ x \in \R,
\end{array}
\right.
\end{equation}
in the sense of distributions, where $u_0$ is an initial 
probability density.
We look for a solution of (\ref{E1.1}) with time evolution in $L^1(\R)$.

We always make  the general following assumption on $\beta$.
\begin{Assumption}\label{E1.0}
\begin{itemize}
\item  $\beta: \R \rightarrow \R$ is
 monotone increasing. 
\item 
$\vert \beta (u) \vert \le  {\rm const} \vert u \vert, \ u \in \R.  $ 
\item  There is $\lambda > 0$ such that
$(\beta \mp  \lambda id)(x) \rightarrow \mp \infty$ 
when $x \rightarrow \mp \infty$
where $id(x) \equiv x$.
\item $u_0 \in (L^1 \bigcap L^\infty)(\R).$
\end{itemize}
\end{Assumption}

\begin{rem} \label{RE1,0bis}
\begin{enumerate}
\item  Since $\beta$ is monotone, (\ref{E1.0}) implies that
   $ \beta (u) = \Phi^2(u) u$, $\Phi$ being a
non-negative bounded Borel function. 
\item $\beta(0) = 0$ and $\beta$ is continuous at zero.
\end{enumerate}
\end{rem}
We recall that when $\beta (u) = \vert u \vert
 u^{m-1}$, $m  >   1$, (\ref{E1.1}) is  nothing else but the
classical {\it porous media equation}.

One of our final targets is to consider $\Phi$
 as continuous except for a possible jump at one positive point,
say $e_c  > 0  $. A  typical example is 
\begin{equation} \label{E1.1a} 
\Phi (u) = H(u-e_c), 
\end{equation}
$H$ being the Heaviside
function. 

The analysis  of (\ref{E1.1}) and its probabilistic representation
 can be done in the framework of monotone partial differential
equations (PDE) allowing multi-valued functions and will be discussed
in detail in the main body of the paper. 
This extension is necessary, among other things, to allow the graph associated
with $\beta$ to be a maximal monotone graph.
We refer to  Assumption
\ref{H3.0} below.
In this introduction, for simplicity, 
we restrict our presentation to the single-valued case.

\begin{defi} \label{DNond} 
We will say that equation (\ref{E1.1}) or $ \beta$ is
{\bf non-degenerate} if there is a constant $c_0 > 0$ such that
$ \Phi  \ge c_0  $. 
\end{defi}
Of course, if $\Phi$ is as in
  (\ref{E1.1a}), then   $\beta$ in is not non-degenerate.
 In order to have $\beta$ to be non-degenerate, one needs to 
add a positive constant to it.

Several contributions were made in this framework starting from
\cite{BeBrC75} for existence,   \cite{BrC79} for  uniqueness in the
case of bounded solutions 
 and \cite{BeC81} 
for continuous dependence on the coefficients.
The authors consider  the case where $\beta$ is  continuous, even
if their arguments allow some extensions for the discontinuous case.

As mentioned in the abstract,  the first motivation 
of this paper was to discuss a continuous time model 
 of self-organized criticality (SOC), which are described by equations
 of type (\ref{E1.1a}).

SOC is a property of  dynamical systems which
 have a critical point as an attractor, see \cite{bak86} for a 
significant monograph on the subject.
SOC is typically observed in slowly-driven out-of-equilibrium 
systems with threshold dynamics relaxing through a 
hierarchy of avalanches of all sizes. We, in particular, refer to
the  interesting physical paper
 \cite{BanJa}.
 The latter  makes reference to a
 system whose evolution is similar to the
 evolution of  a ``snow layer'' under the influence of 
an ``avalanche effect''
which starts when the top of the layer attains a critical value
$e_c$. Adding  a stochastic noise should describe other
contingent effects. For instance, an additive perturbation by noise
could describe the regular effect of ``snow falling''.
In Bantay et al. (\cite{BanJa}) it was proposed to describe this phenomenon 
by a singular diffusion involving precisely a coefficient of the type
\eqref{E1.1a}.

In the absence of noise the density $u(t, x), t > 0, x \in \R$
of this diffusion is formally described by  \eqref{E1.0}
and $\beta (u) = \Phi(u)^2 u$ where $\Phi$ is given by \eqref{E1.1a}.


Such a discontinuous monotone $\beta$ has to be considered as a multivalued map
in order to apply monotonicity methods. 

The singular non-linear diffusion equation  (\ref{E1.1}) models  the {\it macroscopic} phenomenon for
which we try to give a {\it microscopic}  probabilistic
representation, via a non-linear stochastic differential equation
(NLSDE)  modeling the evolution of 
a single point on the layer.

Even if the irregular diffusion equation  (\ref{E1.1}) can be shown to
be well-posed, up to now we can only prove existence (but not yet
uniqueness) of solutions to the
corresponding NLSDE. On the other hand if $\Phi \ge c_0 >  0$, 
then uniqueness can be proved. 
For our applications,  this  will solve the case
$\Phi(u) = H(x - e_c) + \varepsilon $
for some positive $\varepsilon $.
The main novelty with respect to the literature is the fact
that $\Phi$ can be irregular with jumps.


To the best of our knowledge the first author who considered a
probabilistic representation (of the type studied in this paper) for the
solutions of a non-linear deterministic PDE was McKean
\cite{mckean}, particularly in relation with the so called propagation of
chaos. In his case, however, the coefficients were smooth. From then on
the literature steadily grew and nowadays there is a vast amount of
contributions to the subject, especially when the non-linearity is in the
first order part, as e.g. in Burgers equation. We refer the reader to the
excellent survey papers \cite{sznit} and \cite{graham}.

A probabilistic interpretation of (\ref{E1.1}) when 
$\beta(u) = \vert u \vert  u^{m-1}, m   > 1,  $ was provided for
instance in \cite{BCRV}. For the same $\beta$, though  the method
could  be adapted  to the case where $\beta$ is Lipschitz, in
\cite{J00} the author  has studied the evolution equation  (\ref{E1.1})
when the initial condition and the evolution takes values
in the class of  probability distribution functions on $\R$. 
Therefore, instead of an evolution  equation in $L^1(\R)$,  
he considers a state space of
functions vanishing at $- \infty$ and with value $1$ at $+ \infty$.
He  studies both the probabilistic representation
and propagation of chaos.

Let us now describe the principle of the mentioned probabilistic
representation. 
The stochastic differential equation (in the weak sense),
rendering the probabilistic  representation, is given 
by the following (random) non-linear diffusion:
\begin{equation}
\label{E1.2}
\left \{
\begin{array}{ccc}
Y_t &=& Y_0 + \int_0^t \Phi(u(s,Y_s)) dW_s  \\
{\rm Law \quad density } (Y_t) &=& u(t,\cdot), \\
\end{array}
\right.
\end{equation}
where $W$ is a classical Brownian motion.
The solution of that equation may be visualised as a continuous process $Y$
on some filtered probability 
space $(\Omega, \shf, (\shf_t)_{ t \ge 0}, P)$
equipped with a Brownian motion $W$.
  By looking  at a properly chosen version, we can
and shall assume that $Y:[0,T] \times \Omega \rightarrow \R_+ $ 
is $\shb([0,T]) \otimes \shf$-measurable. Of course, we can only
have (weak) uniqueness for (\ref{E1.2}) if we fix the initial
distribution,
 i.e. we have to fix the distribution (density)  $u_0$ of $Y_0$.

The connection with (\ref{E1.1}) is then given by the following
result.

\begin{theo} \label{TI.1}
\begin{description}
\item{(i)}
Let us assume the existence of a solution $Y$ for
(\ref{E1.2}).
Then $u: [0,T] \times \R \rightarrow \R_+$   provides a solution in the sense
of distributions of (\ref{E1.1}) with
$u_0 := u(0,\cdot)$.
\item{(ii)} Let $u$ be a solution of (\ref{E1.1})
in the sense of distributions and let $Y$ solve the first  
equation in (\ref{E1.2}) with law density $v(t, \cdot)$ and initial
law density $u_0 = u(0,\cdot)$. Then
\begin{equation} \label{E1.4prime}
\partial_t v =  \frac{1}{2} \partial_{xx}^2(\Phi^2(u) v),  
\end{equation}
in the sense of distributions. In particular, if $v$ is the unique
solution of  (\ref{E1.4prime}), with $v(0,\cdot) = u_0$, then $v = u$.
\end{description}
\end{theo}
\begin{proof} \
 Let $\varphi \in C_{0}^{\infty}(\R)$, $Y$ be a solution to the first line
  of (\ref{E1.2}) such that $v(t,\cdot)$ is the law density
  $Y_t$, for positive $t$.
 We apply It\^o's formula to $\varphi(Y)$,
to obtain
\begin{equation*}
\varphi (Y_t)= \varphi(Y_0) + \int_{0}^{t} \varphi'(Y_s)
 \Phi(u(s,Y_s)) \,dW_s +\frac{1}{2} \int_ {0}^{t}\varphi''(Y_s) 
 \Phi^2(u(s,Y_s)) \,ds
\end{equation*}
Taking  expectation we obtain 
\begin{equation*}
\int_{\R}\varphi(y)  v(t,y) dy =\int_{\R}\varphi(y) 
 u_{0}(y) dy +\frac{1}{2} \int_ {0}^{t}\,ds\int_{\R}\varphi''(y) 
 \Phi^2(u(s,y)) v(s,y)  \,dy.
\end{equation*}
Now both assertions (i) and (ii) follow.
\end{proof}
\begin{rem}\label{Pos}
\ An immediate consequence of the
probabilistic representation of a  solution of (\ref{E1.1}) is
its positivity at any time.
Also the property that  the initial condition is of mass 1  is in this case conserved.

\end{rem}
The main purpose of this paper is to show existence and uniqueness
in law of the probabilistic representation equation (\ref{E1.2}),
in the case that $\beta$ is non-degenerate 
and not necessarily continuous. 
In addition, we prove  existence for (\ref{E1.2}), in some 
degenerate   cases under certain conditions, see Subsection
\ref{sub4.2}.

Let us now briefly explain
the points that we are able to treat and the difficulties
 which naturally appear in regard to the
probabilistic representation.

For simplicity we do this for
 $\beta$  being
single-valued (and) continuous.
However, with some technical complications this generalizes to
the multi-valued case, as spelled out in the subsequent sections.

\begin{enumerate} 
\item Monotonicity methods allow us to show existence and uniqueness
of solutions to  (\ref{E1.1}) in the sense of distributions
 under the assumption  that $\beta$ is monotone, that
there exists $\lambda > 0$ with $(\beta + \lambda id) (\R) =\R$ and
that $\beta$ is  continuous at zero, 
see Proposition \ref{P3.1} below.
We emphasize that for uniqueness no surjectivity of $\beta + \lambda id$
is required, see Remark \ref{R3.2a} below.
\item Let  $a: [0,T] \times \R \rightarrow \R$
be a   strictly positive bounded Borel function.
Let $\shm(\R)$ be the set of all signed measures on $\R$ with
  finite total variation.
We prove uniqueness of solutions of  
\begin{equation}\label{E1.3}
\left \{
\begin{array}{ccc}
\partial_t v &=&\partial_{xx}^2 
(av)  \\
v(0,x)& =& u_0(x), 
\end{array}
\right.
\end{equation}
as an evolution problem in  $\shm(\R)$, at least under an additional
assumption (A), see  Theorem \ref{P3.5} below. 
\item 
If $\beta$ is {\it non-degenerate}, we can
construct a unique (weak)  solution $Y$ to the non-linear SDE corresponding
to (\ref{E1.2}), for any initial bounded probability density $u_0$ on $\R$, 
see Theorem \ref{T4.2} below.
For this construction, items 1. and 2. above are used in a crucial way.

\item 
Suppose $\beta$ possibly degenerate. 
We fix a bounded probability density $u_0$.
We set $\Phi_\varepsilon = \Phi + \varepsilon$ and 
consider the weak solution  $ Y^\varepsilon$
of 
\begin{equation} \label{E1.2b}
Y^\varepsilon_t =  \int_0^t \Phi_\varepsilon 
(u^\varepsilon(s,Y^\varepsilon_s)) 
dW_s, 
\end{equation}
where $u^\varepsilon (t,\cdot)$ is the law of $Y^\varepsilon_t, t \ge 0$,
and $Y^\varepsilon_0$ is distributed according to $u_0(x) dx$.
The sequence of laws of the  processes $(Y^\varepsilon)$ are tight.
However, the   limiting processes of convergent subsequences  may in
general not solve  the SDE 
\begin{equation} \label{E1.2c}
Y_t =  \int_0^t \Phi (u(s,Y_s)) dW_s. 
\end{equation}
However, under some additional assumptions, see Properties 
\ref{E4.6prime} and  \ref{E4.7prime}
below, it will be the case.
The analysis of the degenerate case
 in greater generality (including
case (\ref{E1.1a})) will be the subject of the forthcoming paper \cite{BRR}.

\end{enumerate}

In this paper, we proceed as follows. Section 2
is devoted to preliminaries about  elliptic PDEs satisfying  
monotonicity conditions. 

In Section 3, we first state a  general existence and uniqueness
result  (Proposition \ref{P3.1}) for equation (\ref{E1.1}) and
 provide its proof, see item 1. above.
The rest of  Section 3
is devoted to the study of the uniqueness of a deterministic,
time inhomogeneous singular linear equation 
with evolution 
 in the space of probabilities on $\R$. This will be applied 
for studying the uniqueness of equation (\ref{E1.3})  in item 2. above. 
This is only possible in the non-degenerate case,
 see Theorem \ref{P3.5}; if $\beta$ is not non-degenerate
   we give  
a counterexample in Remark \ref{RMistake1}.

Section 4 is devoted to the probabilistic representation (\ref{E1.2}).
In particular,  in the non-degenerate
  (however not smooth) case,
 Theorem \ref{T4.2}  gives
  existence and uniqueness  of the non-linear diffusion (\ref{E1.2})
which represents probabilistically (\ref{E1.1}).
In the degenerate case, Proposition \ref{PDeg} 
gives an existence result.

 Finally, we would like to mention that, in order to keep this paper
self-contained and make it accessible to a larger audience, we include 
the analytic background material and necessary (through standard) definitions.
Likewise, we tried to explain all details on the analytic delicate and 
quite technical parts of the paper which form the backbone
of the proofs for our main result.

\section{Preliminaries}

We start with some basic analytical framework.

If $f: \R \rightarrow \R$ is a bounded function we will set
$\Vert f \Vert_\infty  := \sup_{x \in \R} \vert f(x) \vert. $
  By $C_b(\R)$ we denote the space of bounded continuous real functions
and by $\shs\left(  \mathbb{R}\right)  $ the space of rapidly decreasing infinitely
differentiable functions $\varphi:\mathbb{R}\rightarrow
\mathbb{R}$, by $\shs^{\prime}\left(  \mathbb{R}\right)  $ its dual (the
space of tempered distributions).

Let $K_\varepsilon$ be the Green function of $\varepsilon - \Delta$,
that is the kernel of the operator  $(\varepsilon - \Delta)^{-1}:
L^2(\R) \rightarrow L^2(\R)$. So, for all $\varphi \in L^2(\R)$, we have
\begin{equation}
B_\varepsilon(\varphi) := 
   (\varepsilon  -\Delta)^{-1}\varphi  (x)=\int_{\mathbb{R}}%
K_\varepsilon \left(  x-y\right)  \varphi(y)dy.\label{eq kernelbis}%
\end{equation}

The next lemma provides us with an explicit expression of the
 kernel function $K_\varepsilon$. 

\begin{lemma} \label{LEkernel}
\begin{equation} \label{EKernel}
K_\varepsilon \left(  x\right)  =
  \frac{1}{2 \sqrt
  {\varepsilon}} e^{- \sqrt {\varepsilon} \vert x \vert}, \ x \in \R.
\end{equation}
\end{lemma}
\begin{proof} \
From Def. 6.27 in \cite{stein}, we get  
\begin{equation} \label{EKernel1}
K_\varepsilon \left(  x\right)  = \frac{1}{  \left(
4\pi\right)  ^{1/2}}\int_{0}^{\infty}t^{-\frac{1}{2}}e^{-\frac{\left|
x\right|  ^{2}}{4t} - \varepsilon t} dt 
\end{equation}
The result follows by standard calculus.
\end{proof}
Clearly,  if $\varphi \in C^2 (\R) \bigcap \shs'(\R)$,
then $(\varepsilon  - \Delta) \varphi$
 coincides with the classical associated PDE operator.




 \begin{lemma} \label{L3.5}
Let $\varepsilon > 0, \ m \in \shm(\R)$.
There is a unique solution $v_\varepsilon \in C_b(\R) 
\bigcap (\bigcap_{p \ge 1} 
L^p(\R))$  of
\begin{equation} \label{E3.5b}
  \varepsilon v_\varepsilon - \Delta v_\varepsilon = m
\end{equation}
in the sense of distributions given by
\begin{equation} \label{Emvar}
 v_\varepsilon (x) := \int_\R K_\varepsilon (x-y) dm(y), \quad x \in \R.
\end{equation} 
Moreover it fulfills
\begin{equation} \label{E3.5cc}
\sup_x \sqrt \varepsilon \vert v_\varepsilon (x) \vert \le 
\frac{
\Vert m \Vert _{\rm var}}{2},
\end{equation}
where $\Vert m \Vert  _{\rm var} $ denotes the total variation norm.
In addition,
the derivative $ v_\varepsilon'$ has 
a bounded cadlag version which is locally of
bounded variation.
\end{lemma}
In the sequel, in analogy with \eqref{eq kernelbis}, that solution will
 be denoted by $B_\varepsilon m$.

\begin{proof}
\
Uniqueness follows from an obvious application 
of Fourier transform. In fact, it holds even
in $\shs'(\R)$.

 $v_\varepsilon$ given by \eqref{Emvar},
clearly satisfies (\ref{E3.5b}) in the sense of
distributions. By Lebesgue's dominated convergence theorem
and because $K_\varepsilon$ is a bounded continuous function,
 it follows that  $ v_\varepsilon \in C_b(\R)$.

By Lemma \ref{LEkernel} we have 
\begin{equation} \label{E3.5c}
\sup_x \vert v_\varepsilon (x) \vert \le 
\frac{1}{2 \sqrt \varepsilon }  
\Vert m \Vert _{\rm var},
\end{equation}
 By Fubini's
theorem and (\ref{EKernel}),  it follows that $v_\varepsilon \in
L^1(\R)  $.
Hence $v_\varepsilon \in  L^p (\R), \forall p \ge 1$, because
$v_\varepsilon$ is bounded.

Since $ v''_\varepsilon $  equals  $\varepsilon
v_\varepsilon  - m $ 
in the sense of distributions, after integration, we can show that
$$ v_\varepsilon'(x) =
\varepsilon
 \int_{-\infty}^x v_\varepsilon (y) dy -
m(]-\infty, x]),$$
for $dx$-a.e. $ x \in \R$. In particular, $v_\varepsilon$
 has a bounded cadlag version which is locally
 of bounded variation and
$$ \Vert v_\varepsilon' \Vert_\infty \le 
 \varepsilon \Vert v
\Vert_{L^1(\R)} + \Vert m \Vert_{\rm var}.$$
\end{proof}

We now recall  some  basic notions from the analysis
of monotone operators.
More information can also be found 
 for instance in \cite{sho97}. See also  
\cite{B93, Br77}.

Let $E$  be a general Banach space. 

One of the most basic notions of this paper 
is the one of a multivalued function (graph).
A {\bf multivalued function} (graph) $\beta$ on $E$ will be a subset
 of $E \times E$.
It can be seen, either as a family of couples $(e,f), e, f \in E$ and we will
write $f \in \beta(e)$ or as a function $\beta: E \rightarrow \shp(E)$.

We start with a definition in the case $E = \R$.

\begin{defi} \label{D2.4b}
A  multivalued function  $\beta$  defined on $\R$ 
with values in subsets of $\R$ is 
said to be {\bf monotone} if
$(x_1 - x_2) (y_1 - y_2) \ge 0$
for all  $ x_1, x_2 \in \R$,
 $y_i \in \beta(x_i), i = 1,2$.

We say that $\beta$ is {\bf maximal 
monotone} if it is monotone and  there exists $\lambda > 0$ such that
$\beta + \lambda id$ is surjective, i.e.
$$ {\cal R} (\beta + \lambda id) 
:=  \bigcup_{x \in \R} (\beta (x) + \lambda x) = \R.$$
\end{defi}

We recall that one  motivation of this paper is the case
  where $\beta(u) = H(u - e_c) u$.

Let us consider a monotone function $\psi$.
Then all the discontinuities are of jump type.
At every discontinuity point $x$  of $\psi$,
it is possible to {\it complete} $\psi$, producing 
a multi-valued function, 
by setting $\psi(x) = [\psi(x-), \psi(x+)]$.

Since $\psi$ is a monotone function,
the corresponding multivalued   function will be, of course, also
monotone.

Now we come back to the case of our general Banach space $E$ with norm
$\Vert \cdot \Vert.$
An operator $T: E \rightarrow E$ is said to be a {\bf contraction}
 if it is Lipschitz of norm less or equal to 1 and 
 $T(0) =  0$.

\begin{defi}
 A map $ A: E \rightarrow E$, 
or more generally
a multivalued map $A: E \rightarrow  \shp(E)$
is said to be {\bf accretive} if
for all $f_1, f_2, g_1, g_2  \in E$ 
such that $g_i \in A f_i, i = 1,2$,
 we have 
$$ \Vert f_1 - f_2 \Vert  \le  \Vert f_1 - f_2 + \lambda (g_1 - g_2)
 \Vert, $$
for any $\lambda  >    0$.
\end{defi}
This is equivalent to saying the following:  for any $\lambda   >    0 $,
 $(1 + \lambda A)^{-1}$ is 
a contraction for any  $\lambda > 0$ on  $Rg(I +\lambda A)$.
We remark that a contraction is necessarily single-valued.

\begin{rem} \label{R27a}
Suppose that $E$ is a Hilbert space equipped with the scalar product
$(\;,\, )_H$.
 Then  $A$ is accretive if and only if it is {\bf monotone} i.e.
$ ( f_1 - f_2,  g_1 - g_2)_H \ge 0$ for any $f_1, f_2, g_1, g_2  \in E$ 
such that $g_i \in A f_i, i = 1,2$, see  Corollary 1.3 of
\cite{sho97}.
 \end{rem}

\begin{defi} \label{D27a}
 A monotone  map $ A: E \rightarrow E$  (possiblly multivalued)
is said to be 
{\bf m-accretive} if for some $\lambda  >  0 $, 
$A + \lambda I$ is surjective (as a graph in $E \times E$).

\end{defi}
\begin{rem} \label{Rsho}
So, $A$ is m-accretive, if and  only if 
for all $\lambda$ strictly positive, 
$(I + \lambda A)^{-1}$ is a contraction on $E$.

\end{rem}

Now, let us consider the case $E = L^1(\R)$.
The following  is  taken from \cite{BeC81}, Section 1.

\begin{prop} \label{pr1}
 Let $\beta: \R \rightarrow \R$ be a monotone (possibly multi-valued)
map such that  the corresponding graph is m-accretive.
Suppose that $\beta(0) = 0$.

Let  $  f \in E = L^1(\R)$.

\begin{enumerate}
\item
There is a unique $u \in L^1(\R)$ for which there is 
$w \in L^1_{\rm loc}(\R)$ such that 
\begin{equation} \label{pe1} 
u - \Delta w = f \quad {\rm in} \quad 
\shd'(\R), \quad w(x) \in \halb \beta(u(x)), \quad  {\rm for \ a.e. } \quad
x \in \R,
\end{equation}
see Proposition 2 of  \cite{BeC81}.
\item It is then possible to define a (multivalued) operator
$A:=  A_\beta: E \rightarrow E$ where
$ D(A) $  is the set  
of $ u \in L^1(\R) $ for which  there is $w \in L^1_{\rm loc}(\R)$
such that  $w(x) \in \halb \beta(u(x)) $ for \ a.e. \ $x \in \R$ and $\Delta w 
\in L^1 (\R)$. 
For $u \in  D(A)$, we set 
$$ A u = \{ - \halb w \vert w \ {\rm  as \ in \ the \  definition \ of \ }
D(A)\}.$$ 
This is a consequence of the remarks following Theorem 1 in   \cite{BeC81}.

In particular, if $\beta$ is single-valued, $A u = - \halb \Delta \beta(u)$.
We will adopt this notation also if $\beta$ is multi-valued.

\item The operator  $A$ defined in 2. above is m-accretive
on $E = L^1(\R)$, see Proposition 2 of  \cite{BeC81}.

\item
We set  $J_\lambda = (I + \lambda A)^{-1}$, which is a
  single-valued operator.
If $f \in L^\infty (\R) $, then $ \Vert  J_\lambda f \Vert_\infty
  \le \Vert f \Vert_\infty$, 
see Proposition 2 (iii) of  \cite{BeC81}.
In particular, for every  positive integer  $n$,
$ \Vert J_\lambda^n f \Vert_\infty  \le \Vert f \Vert_\infty$.
\end{enumerate}
\end{prop}

Let us summarize some important results of
 the theory of non-linear semigroups, see for
instance \cite{E77, B76, B93, BeBrC75} or the
 more recent monograph \cite{sho97},
which we shall use below.
Let $A: E \rightarrow E$ be a (possibly multivalued) m-accretive operator.
We consider the equation 
\begin{equation} \label{ENonLinEv}
0 \in u'(t) + A (u(t)),  \quad  0 \le  t \le T.
\end{equation}
A function $u: [0,T]  \rightarrow E$, which is absolutely continuous
such that for a.e. $t$, $u(t,\cdot) \in D(A)$ and 
fulfills (\ref{ENonLinEv}) in the following sense, is called
{\bf strong solution}.

There exists $\eta: [0,T] \rightarrow E$, Bochner integrable, such
that $\eta(t) \in A(u(t))$ for a.e. $t \in [0,T]$ and 
$$ u(t) = u_0 + \int_0^t \eta (s) ds, \quad  0  <   t \le T. $$

A weaker notion for (\ref{ENonLinEv}) is the so-called 
{\bf $C^0$- solution}, see Chapter IV.8 of \cite{sho97}.
In order to introduce it, 
 one first defines the notion of $\varepsilon$-solution for
 (\ref{ENonLinEv}).

An {\bf $\varepsilon$-solution} is a discretization 
$$ \shd = \{0= t_0 < t_1 < \ldots < t_N = T \} $$
and an $E$-valued step function 
$$ u^\varepsilon (t) = \left \{
 \begin{array}  {ccc}
 u_0 &: & t= t_0 \\
u_j \in D(A)  &: & t \in ]t_{j-1}, t_j ]
\end{array}
\right.
$$ for which
$ t_j - t_{j-1} \le \varepsilon $ for 
$1 \le j \le N$,
and 
$$ 0 \in \frac{u_j - u_{j-1}}{t_j - t_{j-1}} + A u_j, 1 \le j \le N. $$
We remark that, since $A$ is m-accretive, $u^\varepsilon$ is determined by
$\shd$ and $u_0$, see Proposition \ref{pr1} 1.
\begin{defi} \label{DC0Sol}
A {\bf $C^0$- solution} of (\ref{ENonLinEv}) is a function $u \in C([0,T]; E)$
such that for every $\varepsilon >  0$, there is an
$\varepsilon$-solution  $u^\varepsilon$ of (\ref{ENonLinEv}) with
$$ \Vert u(t) - u^\varepsilon (t) \Vert \le \varepsilon, \quad 0 \le t \le T.$$
 \end{defi}

\begin{prop} \label{Panls}
Let $A$ be an m-accretive (multivalued) operator  on a Banach
space $E$.
We set again $J_\lambda : = (I + \lambda A)^{-1}, \lambda > 0  $.
 Suppose  $u_0 \in \overline{D(A)} $. Then: 
\begin{enumerate}
\item 
 There is a unique  $C^0$- solution $u: [0,T]
\rightarrow E$  of   (\ref{ENonLinEv})
\item  
 $u(t) = \lim_{n \rightarrow \infty} J_{\frac{t}{n}}^n u_0$
uniformly in $t \in[0,T]$.

\end{enumerate}
\end{prop}

\begin{proof} \

1) is stated in
Corollary IV.8.4. of \cite{sho97} and
 2)  is contained in Theorem IV 8.2 of \cite{sho97}.
\end{proof}

The notion of $C^0$-(or mild) solution needs to be introduced
since  the dual  $E^*$ of $E= L^1(\R)$ is not uniformly convex.
If $E^*$ were indeed uniformly convex, we could have stayed with strong
solutions. In fact, according to  Theorem IV 7.1 of \cite{sho97},
for a given $u_0 \in D(A)$, there  would exist a (strong)
solution $u: [0,T] \rightarrow E$ to (\ref{ENonLinEv}),
which is a simpler notion to deal with.
For the comfort of the reader we recall the following properties.
\begin{itemize}
\item A strong solution is a $C^0$-solution,
by Proposition 8.2 of \cite{sho97}.
\item   Theorem 1.2 of
\cite{CE75} says the following.
Given  $u_0 \in \overline{D(A)}$ and given a sequence $(u_0^n)$  in  
 $ D(A)$ converging to $u_0$, 
then, the sequence of the corresponding strong solutions  $(u_n)$  
 converges to the unique  $C^0$-solution of the same equation.
\end{itemize}

%

\section{A porous media equation with singular coefficients}

\setcounter{equation}{0}

In this section, we will provide first an existence and uniqueness 
result for solutions to 
the parabolic deterministic equation (\ref{E1.1}) in the sense
of distributions for multi-valued m-accretive $\beta$. 
The proof is partly  based 
on the theory of non-linear semigroups, see \cite{BeC81} 
for the case when $\beta$ is continuous.

However, the most important result of this section, is an 
 existence and uniqueness result for a ``non-degenerate'' linear
equation for measures, see (\ref{E1.3}).
 This technical result will be crucial for identifying
the law of the process appearing in the probabilistic representation
(\ref{E1.2}).

We suppose that $\beta$ has the same properties as 
those given in the introduction.
However, $\beta$ is allowed to be multi-valued, hence
m-accretive, as a graph, in the sense of Definition \ref{D2.4b}.
Furthermore, generalizing Assumption \ref{E1.0} we shall assume 
the following.
\begin{Assumption} \label{H3.0}
Let $\beta: \R \rightarrow 2^\R$ be an m-accretive graph
with the property  that  there exists  $c > 0$  such that
\begin{equation} \label{E3.0}
 w \in \beta(u) \Rightarrow \vert w \vert \le c \vert u \vert. 
\end{equation}
\end{Assumption}
\begin{rem} \label{RE3.0}
In particular, $\beta(0) = 0$ and $\beta$ is continuous at zero.
 We use again the representation $\beta(u) = \Phi^2 (u) u$ with
$\Phi$ being a non-negative bounded multi-valued map 
$\Phi: \R \rightarrow \R$.
\end{rem}
\begin{rem} \label{R3.0} As mentioned before, if $\beta: \R \rightarrow \R$
is monotone (possibly discontinuous), it is possible to complete
$\beta$ into a monotone graph.
For instance, if $\Phi (x) = H(x-e_c)$, then
 \begin{equation*}
\beta(x) =  
\left \{ 
\begin{array}{ccc}
0      &:& x < e_c \\
 {[0,e_c]}  &:& x = e_c\\
x       &:& x > e_c
\end{array}
       \right.
\end{equation*}
Since the function $\beta$ is monotone, the corresponding graph is monotone.
Moreover $\beta + id$ is surjective so that, by definition, 
$\beta$ is m-accretive.
\end{rem}

\begin{prop}\label{P3.1}
Let $u_0 \in L^1(\R) \bigcap  L^\infty(\R) $
Then there is a unique solution in the sense of distributions 
$ u \in  (L^1 \bigcap L^\infty) ([0,T] \times \R) $ of
\begin{equation} \label{E3.1a}
\left \{
\begin{array}{ccc}
\partial_t u &\in &  \frac{1}{2} \partial_{xx}^2(\beta(u)), \\
u(t,x)& =& u_0(x),  
\end{array}
\right.
\end{equation}
that is, there exists a unique couple
$(u, \eta_u)  \in 
( (L^1 \bigcap L^\infty) ([0,T] \times \R))^2$
such that
\begin{eqnarray} \label{E3.1}
\int u(t,x) \varphi(x) dx &=& \int u_0(x) \varphi(x) dx + \frac{1}{2}
 \int_0^t ds \int \eta_u(s,x) \varphi''(x) dx, \nonumber \\
&&  \forall \varphi \in \shs(\R) \quad {\rm and} \\
\eta_u(t,x) &\in & \beta(u(t,x))\ {\rm for} \ dt \otimes  dx-{\rm a.e.}
 \ (t,x) \in 
    [0,T] \times \R.  \nonumber
\end{eqnarray}
Furthermore, $\Vert u(t, \cdot) \Vert_\infty  \le \Vert u_0 \Vert_\infty$
 for every 
$t \in [0,T]$ and there is a unique version of $u$  such that 
$u \in C([0,T]; L^1(\R))$ ($\subset L^1([0,T] \times \R))$.
\end{prop}

\begin{rem}\label{R3.1}
\begin{enumerate} 
\item 
We remark that, the uniqueness of $u$ determines the uniqueness
 of $\eta \in \beta(u)$ a.e. In fact, for $s, t \in [0,T]$,
we have
\begin{equation} \label{E3.1bis}
 \left(\frac{1}{2} \int_s^t \eta_u(r,\cdot) dr \right)''
= u(t,\cdot) - u(s,\cdot), \ {\rm a.e.}
\end{equation}
Since $\eta_u \in L^1([0,T] \times \R)$, this implies 
 that the function 
$\eta_u$ is  $dt \otimes dx$-a.e. uniquely determined.
Furthermore, since $\beta(0) = 0$ and because $\beta$ is  monotone, 
for $dt \otimes dx$  \  a.e.  \ $(t,x)  \in [0,T] \times \R$
we have 
$$u(t,x) = 0 \Rightarrow \eta_u(t,x) = 0$$
and
$$ u(t,x) \eta_u(t,x) \ge 0.$$ 
\item 
If $\beta$ is continuous then we can take $\eta_u(s,x) = \beta(u(s,x))$.
\item This result applies in the Heaviside case where $\Phi(x) =
  H(x-e_c)$ and in the non-degenerate case  $\Phi(x) =  H(x-e_c) +
  \varepsilon$. 
\end{enumerate}
\end{rem}




\begin{prooff} \ (of Proposition \ref{P3.1}).

We first recall  that by our assumptions, we have
 $(\beta +  \lambda id) (\R) = \R $
for every  $\lambda   > 0$.

\begin{enumerate}
\item The first step is to prove the existence of a $C^0$-solution
of the evolution problem (\ref{ENonLinEv}) in $E = L^1(\R)$, with 
$A$ and $D(A)$ as defined in Proposition \ref{pr1} 2.
Suppose  $\overline{D(A)} = L^1(\R)$.
Then, the existence of a $C^0$-solution 
$ u \in C([0,T];L^1(\R))$
is  a consequence of Proposition
\ref{pr1} 3. and Proposition \ref{Panls} 1.
In particular, $u$  belongs to  $ L^1([0,T] \times \R)$.
\item We  now  prove that
 $ D(A)$  is dense in  in $E = L^1(\R)$.

Let $u \in E$.
 We have to show the existence of a sequence $(u_n)$
 in $D(A)$ converging to $u$ in $E$.
  We set $u_\lambda = (I + \lambda A)^{-1} u$, so that 
 $ u \in   u_\lambda - \lambda \Delta \halb \beta(u_\lambda) $.
  The result follows if we are able to show that
 \begin{equation} \label{Elambda}
 \lim_{\lambda \rightarrow 0} u_\lambda = u, \quad {\rm weakly \ in } \quad E,
 \end{equation}
because then $D(A)$ is weakly sequentially dense in $L^1(\R)$.
In fact we can easily show that $D(A)$ is convex 
and so also its closure.
Hence by Satz 6.12 of \cite{alt} 
$D(A)$ is also weakly sequentially closed 
and so the result would follow.
We continue therefore proving \eqref{Elambda}.
Since $(I + \lambda A)^{-1}$  is a contraction on $E$,  
 $u_\lambda \in E$ and the sequence $(u_\lambda)$ is bounded
 in $L^1(\R)$. Since $u_\lambda \in D(A)$, by definition, there  exists
 $w_\lambda \in L^1_{\rm loc}(\R)$ such that $ w_\lambda(x) \in
 \halb \beta(u_\lambda(x))$ for $dx$-a.e. $x \in \R$, $\Delta w_\lambda
 \in L^1(\R)$ and $u = u_\lambda - \lambda \Delta w_\lambda$.
Since $\beta$ has linear growth,
 $w_\lambda$ also belongs to $E$ for every $\lambda > 0$
  and the sequence $w_\lambda$ is bounded in $E$.
 Consequently, $\lambda w_\lambda$ converges
to zero in $E$ when $\lambda \rightarrow 0$ and it
 follows that  $\lambda \Delta w_\lambda$
converges to zero in  the sense of distributions,
hence $u_\lambda \rightarrow u$ again in the sense of distributions.
Because $(u_\lambda)$ is bounded in $L^1(\R)$, it follows that
 $u_\lambda \rightarrow u$ weakly  in $E = L^1(\R)$,
as $\lambda \rightarrow 0$.

\item The third step consists in showing that a $C^0$-solution
is a solution in the sense of distributions of (\ref{E3.1a}).

Let $\varepsilon > 0$ and consider a family $u^\varepsilon: [0,T]
\rightarrow E$ of $\varepsilon$-solutions.
Note that for $u_0^\varepsilon := u_0$ and for $ 1 \le j \le N$,  
 with $A$ as in Proposition \ref{pr1} 2., we recursively have
\begin{equation} \label{E3.4prime}
 u_j^\varepsilon = (I-(t_j^\varepsilon - t_{j-1}^\varepsilon) A)^{-1}
u_{j-1}^\varepsilon,
\end{equation}
hence
$$ \Delta w_j^\varepsilon = - 
\frac{u_j^\varepsilon -u_{j-1}^\varepsilon}
{t_j^\varepsilon - t_{j-1}^\varepsilon} $$
for some $w_j^\varepsilon \in L^1_{\rm loc}(\R)$
such that $w_j^\varepsilon \in \frac{1}{2} \beta(u^\varepsilon_j), \ dx$ -a.e.
Hence,  for $t \in ]t^\varepsilon_{j-1}, t^\varepsilon_j], $ we have
$$ u^\varepsilon (t,\cdot) = u^\varepsilon 
(t^\varepsilon_{j-1}, \cdot) + 
\int_{t^\varepsilon_{j-1}}^{t^\varepsilon_j} \Delta w^\varepsilon(s, \cdot) ds.$$
where 
$  w^\varepsilon (t) = w_j^\varepsilon,
\quad t \in ]t^\varepsilon_{j-1}, t^\varepsilon_j].$

Consequently, summing up, then for
 $t \in ]t^\varepsilon_{j-1}, t^\varepsilon_j]$,
$$ u^\varepsilon (t,\cdot) = u_0 + 
\int_0^{t} \Delta w^\varepsilon(s, \cdot) ds 
+  (t^\varepsilon_j - t)  \Delta w^\varepsilon (t^\varepsilon_j, \cdot).$$  
We integrate against  a test function $\alpha \in \shs(\R)$ and get
\begin{eqnarray} \label{EDist}
 \int_\R u^\varepsilon (t,x) \alpha(x) dx &=& \int_\R u_0(x)
\alpha(x) dx +
\int_0^{t} \int_\R  w^\varepsilon(s, x) \alpha''(x) dx  ds 
\nonumber  \\
&& \\
&+& (t - t^\varepsilon_j) \int_\R  w^\varepsilon(t^\varepsilon_j,x)
 \alpha''(x) dx.
 \nonumber
\end{eqnarray}

Letting $\varepsilon$ go to zero we use the fact that 
 $u^\varepsilon \rightarrow u$ uniformly in $t$ in $L^1(\R)$.
$(u^\varepsilon)$ converges in particular to $u \in L^1([0,T] \times \R)$
when $\varepsilon \rightarrow 0$.

The third term in the right-hand side of (\ref{EDist}) converges to zero 
since $ t - t^\varepsilon_j$
is smaller than the mesh $\varepsilon$ of the subdivision.

Consequently, (\ref{EDist}) implies 
\begin{equation} \label{EDist1}
 \int_\R u (t,x) \alpha(x) dx = \int_\R u_0(x) \alpha(x) dx +
\lim_{\varepsilon \rightarrow 0}  \int_0^{t} \int_\R  w^\varepsilon(s,x)
 \alpha''(x) dx  ds.
\end{equation}

According to our assumption on  $\beta$, there is a constant $c > 0$ such that
$\vert w^\varepsilon \vert  \le c \vert u^\varepsilon \vert.$
Therefore the sequence $(w^\varepsilon)$ is equi-integrable on
 $[0,T] \times \R$. So, there is a sequence $(\varepsilon_n)$ such that
$w^{\varepsilon_n}$ converges to some 
$\frac{1}{2} \eta_u \in L^1([0,T]\times \R)$
in  $ \sigma(L^1,L^\infty)$.
Taking  (\ref{EDist1}) into account, it remains to
 see that $\eta_u(t,x) \in \beta(u(t,x))$ a.e. $dt \otimes dx$, 
in order to prove that $u$ solves (\ref{E3.1}).
 
Let $K > 0$.
Using Proposition \ref{pr1} 4., by (\ref{E3.4prime}) we conclude that
$ \Vert u^\varepsilon(t, \cdot) \Vert_\infty \le \Vert u_0 \Vert_\infty$.
Consequently for any $K > 0$, the dominated convergence theorem, 
 implies  that the sequence $u^{\varepsilon_n}$ restricted
to $[0,T] \times [-K,K]$ converges to $u$ restricted to $[0,T] \times [-K,K]$
in $L^2([0,T] \times [-K,K])$ and $w^{\varepsilon_n}$ restricted
to $[0,T] \times [-K,K]$, being bounded by $c \vert u^{\varepsilon_n} \vert$,
 converges (up to a subsequence) 
 weakly in $L^2$,
 necessarily to $\halb \eta_u$ restricted to  $[0,T] \times [-K,K]$. 
The map $v \rightarrow \halb \beta(v)$ on $L^2([0,T] \times [-K,K])$
is an $m$-accretive multi-valued map,  see \cite{sho97}, p. 164, Example 2c.
So it is weakly-strongly closed
 because of
\cite{B93} p. 37, Proposition 1.1 (i) and (ii). 
Hence  the result follows.

\item The fourth step consists in showing that the obtained
solution is in $L^\infty ([0,T] \times \R)$.

Point 2. of  Proposition \ref{Panls}  tells us
that 
$$ u(t, \cdot) = \lim_{n \rightarrow +\infty} J_{\frac{t}{n}}^n u_0 $$
in $L^1(\R)$. 
Hence, for every $t \in ]0,T]$ and for some subsequence $(n_k)$
 depending on $t$,
$$ \vert u(t, \cdot) \vert = \lim_{k \rightarrow \infty} \vert 
J^{n_k}_{\frac{t}{n_k}} u_0 \vert \le  \Vert u_0 \Vert_\infty,
 \quad dx{\rm -a.e.}, $$  
where we used again Proposition \ref{pr1} 4). It follows by Fubini's theorem
that $\vert u(t,x) \vert \le \Vert u_0 \Vert_\infty,
 \quad  {\rm for} \quad dt \otimes dx$-a.e. 
$(t,x) \in [0,T] \times \R$.

\item Finally, uniqueness of the equation in $\shd'([0,T] \times \R)$
follows from
Theorem 1 and  Remark 1.20 of  \cite{BrC79}.
\end{enumerate}
\end{prooff}
\begin{rem} \label{R3.2a}
\begin{enumerate}
\item Theorem 1 and  Remark 1.20 of  \cite{BrC79} apply 
if $\beta$ is  continuous, to give the uniqueness in point 5. above.
 However,  Remark 1.21 of  \cite{BrC79}  says that this
holds true even if  $\beta(0) = 0$ and $\beta$   is only  continuous
in zero and possibly
 multi-valued.
This case applies for instance  when $\Phi(x) = H(x - e_c), \ e_c > 0$.
\item 
We would like to mention that there are variants of the  results
in Proposition \ref{P3.1}
known from the literature.
However, some of them are just 
for bounded domains while we work in all of $\R$. 
For instance,  when the domain is bounded  and $\beta$ is continuous,
 Example 9B in Section IV.9 of \cite{sho97},
  remarks that a $C^0$-solution is a solution in 
the sense of distributions.
\end{enumerate}
\end{rem}

In order to establish the well-posedness  for the related
probabilistic representation one  needs 
 a uniqueness result  for the  evolution of probability measures.
This will be the subject of Theorem \ref{P3.5} below.
But as will turn out, it will require
 some global $L^2$-integrability for the solutions.

A first step in this direction was Corollary 3.2 of \cite{BKR},
that we quote here for the convenience of the reader.

\begin{lemma} \label{LBKR} Let $\kappa \in ]0,T[$.
Let $\mu$ be a finite Borel measure on
$[\kappa ,T] \times \R$; let $a, b \in L^1 ([\kappa,T] \times \R; \mu)$.
We suppose that
$$ \int_{[\kappa,T] \times \R} \left (\partial_t \varphi(t,x)  + a (t,x) 
\partial^2_{xx} \varphi (t,x)  + b(t,x)
 \partial_{x} \varphi (t,x)   \right) \mu(dt dx) =0, $$
for all $\varphi \in C^\infty_0 (]0,+ \infty[ \times \R)$.
Then, there is  $\rho \in L^2_{\rm loc}([\kappa,T] \times \R)$ 
such that 
$$ \sqrt {a(t,x)}  d\mu (t,x) = \rho (t,x) dt dx.$$
 \end{lemma}

We denote the subset of positive  measures in  $\shm(\R)$
by $\shm_+ (\R)$.

\begin{theo} \label{P3.5}
Let $a$ be a Borel non negative bounded  function
on $[0,T] \times \R$. Let  $z_i : [0,T] \rightarrow \shm_+(\R)$,
$i = 1,2$,  be continuous 
with respect to the weak topology of finite measures on $\shm(\R)$.

Let $z^0$ be an element of $\shm_+(\R)$. 
Suppose that both $z_1$ and $z_2$ solve  the problem 
$\partial_t z = \partial_{xx}^2 (a z) $
in the sense of distributions  with initial condition
$z(0) = z^0$.

More precisely, 
\begin{equation} \label{E3.5a}
\int_\R  \varphi(x) z(t)(dx) = \int_\R 
 \varphi(x)  z^0(dx) + \int_0^t ds \int_\R \varphi''(x)  a(s,x) z(s)(dx)
\end{equation}
for every $t \in [0,T]$ 
and any $\varphi \in \shs (\R) $.

Then $(z_1 -  z_2) (t)$ is identically  zero for 
 every $t$, if 
 $z:= z_1 - z_2$, 
 satisfies the following:
\end{theo}
 {\bf ASSUMPTION (A)}:
 There is $\rho: [0,T] \times \R \rightarrow  \R$ belonging to 
$L^2 ([\kappa,T] \times \R) $ for every $\kappa > 0$ such that 
$\rho(t,\cdot)$ is the density of $z(t)$ for almost all $t \in ]0,T]$.
\begin{rem} \label{R3.7bis}
If $a \ge {\rm const} > 0$, then $\rho$ such that
$\rho(t,\cdot)$ is a density of  $(z_1-z_2)(t)$ 
 for almost all $t > 0$,
always exists, via Lemma \ref{LBKR}.
 It remains to check if it is indeed square integrable
on every $[\kappa, T] \times \R$.
\end{rem}

\begin{rem} \label{R3.5bis}
The weak continuity of $z(t, \cdot)$ implies that
$$  \sup_{t \in [0,T]} \Vert z(t) \Vert_{\rm var} < \infty.$$
 Indeed, if this were not true, we could find $t_n \in
[0,T]$, such that $\Vert z(t_n) \Vert_{\rm var} $ diverges to infinity.
We may assume that $\lim_{n \rightarrow \infty} t_n = t_0 \in [0,T] $.
Then $$ \lim_{n \rightarrow \infty} \int_\R f(x) z(t_n)( dx) = 
\int_\R f(x) z(t_0)(dx)$$
for all $ f \in C_b(\R) $, hence by the uniform boundedness principle
one gets the contradiction  that 
$$ \sup_n \Vert z(t_n) \Vert_{\rm var} < \infty.$$
\end{rem}

\begin{rem} \label{RMistake1}
 Theorem \ref{P3.5} does not hold 
without  Assumption (A) even in the time-homogeneous case.

To explain this, let  $\Phi : \R \rightarrow \R_+$  be continuous
and bounded  such that
$\Phi(0) = 0$ and $\Phi$ is strictly positive on $\R-\{0\}$.
We also suppose that $\frac{1}{\Phi^2}$ is integrable in a
neighborhood of zero.

We choose  $z^0 := \delta_0$, i.e. the Dirac measure at zero. 
 It is then possible to exhibit two different solutions
to the considered problem with initial condition $z^0$.

We justify this in the following lines
using a probabilistic representation.
Let $Y_0$ be identically zero.

According to the Engelbert-Schmidt criterion, see e. g.
  Theorem 5.4 and Remark 5.6 of Chapter 5,
\cite{ks},
it is possible to construct two solutions (in law) to the SDE
\begin{equation} \label{SDEES}
 Y_t =   \int_0^t \Phi(Y_s) dW_s.
\end{equation}
where $W$ is a Brownian motion on some filtered probability space.

One solution $Y^{(1)}$ is identically zero. The 
second one $Y^{(2)}$ is 
 a non-constant martingale starting from zero.
We recall the construction of $Y^{(2)}$, since it is of independent  interest.

Let $B$ be a classical Brownian motion and we set
\begin{equation} \label{(3.10prime)}
T_t = \int_0^t \frac{du}{\Phi^2(B_u)}.
\end{equation}
Problem 6.30 of \cite{ks} says that the increasing process
$(T_t)$ diverges to infinity when $t$ goes to infinity.
We define pathwise $(A_t)$ as the inverse of $(T_t)$
and we set $M_t = B_{A_t}$. $M$ is a martingale since
it is a time change of Brownian motion.
One the one hand we have $[M]_t = A_t$.
But pathwise, by \eqref{(3.10prime)}
 we have
$$ A_t = \int_0^{A_t} \Phi^2(B_u) dT_u = 
\int_0^t \Phi^2(B_{A_v}) dv, $$
through a change of variables $u = A_v$.
Consequently we get
$$ A_t = \int_0^t \Phi^2(M_v) dv.$$
Theorem 4.2 of Ch. 3 of \cite{ks}
says that there is a Brownian motion $\tilde W$
on a suitable filtered larger probability space
and an adapted process $(\rho_t)$
so that  $M_t = \int_0^t \rho d \tilde W.$
We have $[M]_t = \int_0^t \rho_s^2 ds =  \int_0^t \Phi^2(M_s) ds,$
for all $t \ge 0$, hence $\rho_t^2 = \Phi^2(M_t)$ and so
$\Phi(M_t) {\rm sign}(\rho_t) = \rho_t$. 

We define 
$$  W_t  = \int_0^t {\rm sign}(\rho_v) d\tilde W_v.  $$
Clearly $[W]_t = t$. By L\'evy's characterization theorem of Brownian motion,
$W$ is a standard Brownian motion. Moreover,     
we obtain $M_t = \int_0^t \Phi(M_s) dW_s$
so that $Y^{(2)} := M$ solves the stochastic differential equation
(\ref{SDEES}).
Now $Y^1_t$ and $Y^2_t$ have not the same marginal laws
$v_i(t,\cdot), i = 1,2 $. In fact $v_1(t,\cdot)$ is
equal to $\delta_0$ for all $t \in  [0,T]$.

Using It\^o's formula it is easy to show that
the law $v(t,\cdot)$ of a solution $Y$ of (\ref{SDEES})
solves the PDE  in Theorem \ref{P3.5} with
$a:= \Phi^2$ and
 initial condition
$\delta_0$. This constitutes a counterexample to 
Theorem  \ref{P3.5}  without Assumption (A).
\end{rem}

\begin{prooff} 

(of  Theorem  \ref{P3.5}).

The arguments developed in this proof
 is inspired by a uniqueness proof of
distributional solutions for the porous media equation,
see Theorem 1 of \cite{BrC79}.

Given a locally integrable function $(t,x) \rightarrow u(t,x)$,
$u'$ (resp. $u''$) stands for the first (resp. second)
 distributional derivative 
with respect to the second variable $x$. 

In the first part of the proof we do not use Assumption (A). 
We will   explicitly state  from where it is needed.

Let $z^1, z^2$ be two solutions
 to (\ref{E3.5a}) and we set $z = z^1-z^2$.
We will  study the quantity
$$ g_\varepsilon (t) = \int_\R B_\varepsilon z (t)(x) z(t)(dx), $$
where $B_\varepsilon z(t) \in (L^1 \bigcap L^\infty) (\R) $ is the continuous
 function $v_\varepsilon$
 defined in Lemma \ref{L3.5}, taking $m = z(t)$.
$g_\varepsilon (t)$ is well-defined, since
$$ g_\varepsilon (t) \le \Vert z(t) \Vert_{\rm var}
 \sup_x \vert B_\varepsilon z
 (t)(x)\vert  \ {\rm for \ all} \ t \in [0,T].   $$

Assume we can show  that 
\begin{equation}
\label{E3.5d} 
\lim_{\varepsilon \to 0} g_\varepsilon (t) = 0
 \ {\rm for \ all} \ t \in [0,T].
\end{equation}
Then we are able to prove  that $z(t) \equiv 0$  for all $t \in [0,T]$.

Indeed, Lemma \ref{L3.5}  says that $B_\varepsilon z(t)' $ is  bounded,
with a version locally of
 bounded variation and that  $B_\varepsilon z(t) 
\in C_b(\R) \bigcap    L^p(\R)$ for all $p \ge 1$. 

Let now  $\shc, \tilde \shc$ be  positive real constants. 
 Then, since all terms in
(\ref{E3.5b}) are signed measures of finite total variation,  (\ref{E3.5b})
implies that 
\begin{eqnarray} \label{E3.5bis}
\int_{]- \tilde \shc, \shc]} B_\varepsilon z(t)(x) z(t)(dx) &=& \varepsilon 
\int_{]- \tilde \shc, \shc]}  (B_\varepsilon z(t)(x))^2 dx \nonumber \\
&& \\
&-& \int_{]- \tilde \shc, \shc]}  B_\varepsilon z(t)(x)  B_\varepsilon
z(t)''(dx). \nonumber
\end{eqnarray}

If $F, G$ are functions of  locally bounded variation, 
$F$ continuous, $G$ right-continuous, 
classical Lebesgue-Stieltjes calculus implies that
\begin{equation} \label{E3.5ter}
 \int_{]- \tilde \shc, \shc]} F dG = FG(\shc) - FG(- \tilde \shc) 
- \int_{] -\tilde \shc, \shc]} GdF.
\end{equation}
Setting $F =  B_\varepsilon z(t),  G(x) = B_\varepsilon z(t)'$,
we get 
$$ - \int_{]- \tilde \shc, \shc]}   B_\varepsilon z(t)(x)
  B_\varepsilon z(t)''(dx) 
= - B_\varepsilon z(t)(\shc)  B_\varepsilon z(t)'(\shc) 
+ B_\varepsilon z(t)(- \tilde \shc)  B_\varepsilon z(t)'( - \tilde \shc)
$$
$$
+  \int_{ ]-\tilde \shc, \shc]}  (B_\varepsilon z (t)'(x))^2 dx.$$
Since $ B_\varepsilon z(t) \in L^1(\R)$, we can choose sequences $(\shc_n),
(\tilde \shc_n) $ converging to infinity
 such that $  B_\varepsilon z(t)(\shc_n) \rightarrow  0, 
\quad   B_\varepsilon z(t)(-\tilde \shc_n) \rightarrow  0$ as 
$n \rightarrow \infty $.
 Then, letting $n \rightarrow \infty$  and using the fact that
 $B_\varepsilon z(t)$  and  $B_\varepsilon z(t)'$ are
bounded,  by the  monotone and Lebesgue  dominated
convergence theorems, we conclude that 
$$ -  \int  B_\varepsilon z(t)(x)  B_\varepsilon z(t)''(dx) 
=   \int  (B_\varepsilon z(t)'(x))^2 dx.$$
In particular, $ B_\varepsilon z(t)' \in L^2(\R) $.
 Consequently, (\ref{E3.5bis}) implies that
\begin{eqnarray*}
g_\varepsilon (t) &=&  \int B_\varepsilon z(t)(x) z(t)(dx) = \varepsilon 
\int  (B_\varepsilon z(t)(x))^2 dx \\
&+&  \int  (B_\varepsilon z(t)'(x))^2 dx.
\end{eqnarray*}
In particular, the left-hand side is positive.
Therefore, if for all $t \in [0,T]$, $g_\varepsilon (t) \rightarrow 0$,
 as $\varepsilon
\rightarrow 0$, then
\begin{eqnarray*}
\sqrt \varepsilon B_\varepsilon z(t) & \rightarrow & 0 \\
 B_\varepsilon z(t)' & \rightarrow & 0
\end{eqnarray*}
in $L^2(\R)$, as $\varepsilon \rightarrow 0$,
and so, for all $t \in [0,T]$,
$$ z(t) =  \varepsilon B_\varepsilon z(t)  - 
B_\varepsilon z(t)'' \rightarrow 0 $$
in the sense of distributions.
Therefore,  $z \equiv 0$. \\
It remains to prove (\ref{E3.5d}). \\
Let $\delta > 0$ and $\phi_\delta \in C^\infty_\circ(\R), \phi_\delta
\ge 0$, symmetric, with $\int_\R \phi_\delta(x) dx = 1$ weakly
approximating the Dirac-measure with mass in $x = 0$. Set
$$ z_\delta(t,x) := (\phi_\delta \star z(t))(x) := \int_\R \phi_\delta(x-y)
 z(t)(dy), \quad  x \in \R, t \in [0,T].$$
We define $h: [0,T] \rightarrow \shm(\R)$ by $h(t) (dx) = a(t,x) z(t,dx)$.
Note that by (\ref{E3.5a}), since $\phi_\delta(x- \cdot) \in \shs(\R),
\forall x \in \R$, we have
\begin{equation} \label{E3.13}
 z_\delta (t,x) = \int_0^t \int_\R \phi''_\delta(x-y) h(s)(dy) ds = 
\int_0^t (\phi''_\delta \star h(s))(x) ds, \forall t \in [0,T], x \in
\R, 
\end{equation}
where we used that $z_\delta (0) = 0$, because $z(0) = 0$, and 
that $x \mapsto z_\delta(t,x) $ is continuous for all $t \in 
[0,T]$. In fact, one
can easily prove that $z_\delta$ is continuous and bounded on $[0,T]
\times \R$. \\
Let us now consider $w \in \shs(\R)$.  
By 
 Fubini's theorem, for all $t \in  [0,T] $
it follows that
\begin{eqnarray*}
\int_\R w(x) B_\varepsilon z(t)(x) dx &=&
 \int_\R w(x) \int_\R K_\varepsilon (x-y) z(t)(dy) dx \\
&=& \int_\R (w \star K_\varepsilon)(y) z(t)(dy).
\end{eqnarray*}
Now, $B_\varepsilon z(0) = 0 $ since $z(0) = 0$.
Therefore
by \eqref{E3.5a}, and the fact that 
$w \star K_\varepsilon \in \shs(\R)$,
the previous expression  is equal to
$$ \int_0^t  \int_\R (w \star K_\varepsilon)''(y) h(s)(dy) ds
=  \int_0^t \int_\R w''(x) 
 B_\varepsilon h(s) (x) dx ds, $$
which in turn by Lemma \ref{L3.5}  is equal to
$$  \int_0^t \int_\R w(x) (
\varepsilon B_\varepsilon h(s) (x) dx - h(s)(dx))ds.$$
Consequently, by approximation,
\begin{eqnarray} \label{E3.14}
\int_\R w(x) B_\varepsilon z(t)(x) dx &=& \int_0^t \int_\R w(x) (
\varepsilon B_\varepsilon h(s) (x) dx - h(s)(dx) ) ds \nonumber \\
&&\\
& \forall & w \in C_b(\R), t \in [0,T].\nonumber
\end{eqnarray}

As a consequence of (\ref{E3.13}) and (\ref{E3.14})
 and again using Fubini's theorem, for all $t \in [0,T]$
we obtain
\begin{align*}
 g_{\varepsilon, \delta} (t) &:= \int_\R z_\delta (t,x)
B_\varepsilon z(t)(x) dx \\
& \underbrace{=}_{ (\ref{E3.14})}  \int_0^t \int_\R  z_\delta (t, x)
(  \varepsilon B_\varepsilon h(s) (x) dx - h(s)(dx) ) ds \\
& \underbrace{=}_{(\ref{E3.13})}   
  \int_0^t \int_\R  z_\delta (s, x)
(  \varepsilon B_\varepsilon h(s) (x) dx - h(s)(dx) ) ds \\ &+
\int_0^t \int_\R \int_s^t ( \phi_\delta'' \star h(r))(x) dr ( \varepsilon
 B_\varepsilon h(s) (x) dx - h(s)(dx)) ds  \\
&= \int_0^t \int_\R  z_\delta (s, x)
(  \varepsilon B_\varepsilon h(s) (x) dx - h(s)(dx) ) ds\\ & +
\int_0^t  \int_0^r \int_\R ( \phi_\delta'' \star h(r))(x) 
( \varepsilon
 B_\varepsilon h(s) (x) dx - h(s)(dx)) ds dr \\
\end{align*}
\begin{align*}
 \phantom{g_{\varepsilon, \delta} (t)}  &
 \underbrace{=}_{(\ref{E3.14})}
\int_0^t \int_\R  z_\delta (s, x)
(  \varepsilon B_\varepsilon h(s) (x) dx - h(s)(dx) ) ds \\
&+
\int_0^t   \int_\R ( \phi_\delta'' \star h(r))(x)  
 B_\varepsilon z(r) (x) dx  dr
 \end{align*}
The application of Fubini's theorem above is justified since $a$ is
bounded, $\sup_{t \in [0,T]} \Vert z(t) \Vert_{\rm var} < \infty,$ 
$K_\varepsilon$ is bounded and $ \phi_\delta \in \shs(\R) $. 
But the last term is equal to 
 \begin{eqnarray*}
 &&\int_0^{t} \int_\R  \int_\R \phi_\delta'' (x-y) 
  B_\varepsilon z(r) (x) dx  h(r)(dy) dr \\
&=& \int_0^t \int_ \R \int_\R   \phi_\delta  (x-y) )
(\varepsilon  B_\varepsilon z(r) (x) dx - z(r)(dx))  h(r)(dy) dr \\
&=&\int_0^t \int_\R \varepsilon B_\varepsilon z(r)(x) (\phi_\delta \star
  h(r))(x) dx dr - \int_0^t \int_\R z_\delta (r,y) h(r)(dy) dr, 
 \end{eqnarray*}
where we could use Lemma \ref{L3.5} in the first step, since
$\phi_\delta  (\cdot -y) \in \shs(\R), \forall y \in \R$. 
Hence, for all $t \in [0,T]$,
\begin{eqnarray} \label{E3.15}
 g_{\varepsilon, \delta} (t) &=& \int_0^t \int_\R z_\delta(s,x)
\varepsilon B_\varepsilon h(s)(x) dx ds  \nonumber \\
&+& \int_0^t \int_\R \varepsilon B_\varepsilon z(s) (x) (\phi_\delta \star
h(s))(x) dx ds  \\
& -& 2\int_0^t \int_\R z_\delta (s,x) h(s)(dx) ds. \nonumber
\end{eqnarray}
For a signed measure $\nu$, we denote its absolute value by $\vert
\nu \vert$.
By Lemma \ref{L3.5} we have 
$$ \sup_{s \in [0,T]} \int_\R (\vert z(s)\vert \star \phi_\delta)(x)
\varepsilon B_\varepsilon \vert h (s) \vert (x) dx
 \le  C \sqrt \varepsilon, $$
where 
$$ C =  \frac{1}{2} \Vert a \Vert_\infty \sup_{s \in [0,T]}
\Vert z(s) \Vert_{\rm var}^2, $$ 
and likewise the integrand of the second integral in (\ref{E3.15})
 is bounded by the
same constant independent of $\delta$.
Hence, as $\varepsilon \rightarrow 0$, the first and second term in
the right-hand side of (\ref{E3.15}) converges to zero uniformly in
$\delta$ and uniformly in $t \in [0,T]$. Now, we use Assumption (A),
namely that $z \in L^2([\kappa, T] \times \R)$ for all $\kappa > 0$.
 Then, since  $B_\varepsilon z(t) \in L^2(\R), \forall t \in [\kappa, T]$, 
and $\Vert a \Vert_\infty < \infty$, (\ref{E3.15}) implies that $\forall
\kappa > 0, t\in [\kappa,T],$ 
\begin{eqnarray} \label{E3.16}
g_\varepsilon (t) - g_\epsilon (\kappa) &=& \lim_{\delta \rightarrow 0}
(g_{\varepsilon, \delta} (t) - g_{\varepsilon, \delta} (\kappa))  \nonumber \\
 &\le & 2 \sqrt {\varepsilon} T C - 2 \int_\kappa^t  \int_\R z^2(s,x)
a(s,x) dx ds \\  
&\le &  2 \sqrt {\varepsilon} T C. \nonumber
\end{eqnarray}

Now, $\lim_{\kappa \rightarrow 0} g_\varepsilon (\kappa) = 0$.
In fact
 $z (\kappa, \cdot) \rightarrow z (0, \cdot) = 0$
weakly, according to the assumption of Theorem \ref{P3.5}.
According to  Theorem 8.4.10, page 192, of 
\cite{bogachev}, the tensor product 
 $z (\kappa, \cdot) \otimes z (\kappa, \cdot)$ converges weakly to
zero. 
On the other hand $(x,y) \mapsto K_\varepsilon (x-y) $
is bounded and continuous on $\R^2$. By Fubini's theorem
$$  g_\varepsilon (\kappa) = \int_{\R^2} z(\kappa)(dx) z(\kappa)(dy)
K_\varepsilon (x-y) \rightarrow 0.$$

So, letting first $\kappa \rightarrow 0$  in (\ref{E3.16}) and then
$\varepsilon \rightarrow 0$, (\ref{E3.5d}) follows since $g_\varepsilon
(t) \ge 0$ for all $t \in [0,T]$. In fact, we even proved that the
convergence in  (\ref{E3.5d})  is uniformly in $t \in [0,T]$.
\end{prooff}
\begin{rem}\label{figalli}
Since our coefficient in Theorem 3.6 is only measurable and possibly
degenerate, to the best of our knowledge this result is really new. 
For instance, in recent contributions by \cite{bris, figalli}, 
the diffusion coefficient is supposed to satisfy at least Sobolev
regularity.

\end{rem}

Theorem \ref{P3.5}   will be useful
for the probabilistic representation of the solution of (\ref{E3.1})
when $\beta$ is non-degenerate.

\section{The probabilistic representation of the deterministic
 equation}

\setcounter{equation}{0}

Despite the fact that $\beta$ is multi-valued, by its monotonicity and 
because of (\ref{E3.0}), it is still possible to find a multi-valued map 
$\Phi: \R \rightarrow \R_+$ such that 
$$ \beta(u) = \Phi^2(u) u, \quad u \in \R, $$
which is bounded, i.e. 
$$ \sup_{u \in \R_*} \sup \Phi(u) < \infty. $$
In fact the value of $\Phi$ at zero is not determined
by $\beta$.

We start with the case where $\Phi$ is non-degenerate.
The value $\Phi(0)$ being a priori arbitrary, we can
 set 
$$\Phi(0) = [\liminf_{u \rightarrow 0} \inf \Phi(u),
 \limsup_{u \rightarrow 0} \sup \Phi(u)]. $$
\begin{defi} \label{DNondG}
The (possibly) multivalued map $\beta$  
(or equivalently $\Phi$) is called {\bf non-degenerate},
if 
 there exists  some constant $c_0 > 0$
such that $ y \in \Phi(u) \Rightarrow y \ge c_0$
for any $u \in \R$.      
\end{defi} 

\subsection{The non-degenerate case}

We  suppose in this subsection $\beta$ to be non-degenerate.

First of all we need to show that solutions 
of the linear PDE (\ref{E3.5a}), which are laws of  solutions to an
SDE, are   space-time square integrable.

\begin{prop} \label{Psqintd}
Suppose $ a: [0,T] \times \R \rightarrow \R $ to be a bounded measurable 
function which is bounded below 
on any compact set by a strictly positive constant.

 We consider a stochastic process  $Y = (Y_t, t \in
[0,T])$ on a stochastic basis $(\Omega, \shf, (\shf_t), P)$,
 being a weak solution of the SDE
$$ Y_t = Y_0 + \int_0^t \sqrt{2 a(s, Y_s)}dW_s,$$
where $W$ is a standard $(\shf_t)$-Brownian motion.
For $t \in [0,T]$, let 
 $z(t)$ be the 
law of $Y_t$ and set $z^0  := z(0)$.
\begin{enumerate}
\item Then $z$ solves equation  (\ref{E3.5a}) with $z^0$ as initial condition.
\item There is $\rho \in L^2([0,T] \times \R)$
 such that $\rho(t, \cdot)$ is the density
    of $z(t)$ for almost all $t \in [0,T]$.  
\item $z$ is the unique solution of (\ref{E3.5a}) with initial condition $z^0$ 
having the property described in item 2. above.
\end{enumerate}
\end{prop}
\begin{rem} \label{rsv}
 A necessary and sufficient condition
 for the existence and uniquess  in law of solutions for the equation in
Proposition \ref{Psqintd}, is that  $Y$ solves the martingale 
problem of Stroock-Varadhan, see Chap. 6 of \cite{SV79}, related to
$ L_t f =  a(t,x) f''.$
In our case, existence and uniqueness follow
 for instance from \cite{SV79}, Exercises 7.3.2-7.3.4, see
also \cite{ks}, Refinements 4.32, Chap. 5.
We remark that the coefficients are not continuous but only
measurable, so that space dimension 1 is essential. 

The reader can also consult 
\cite{RS93,S93} for more refined conditions to
be able to construct a weak solution; however those
do not apply in our case.

\end{rem}
\begin{prooff} \ (of the Proposition \ref{Psqintd}).

\begin{enumerate}
\item The first point follows from a direct application of
It\^o's formula to $\varphi(Y_t)$, $ \varphi \in \shs(\R)$,
cf. the proof of Theorem \ref{TI.1}.
\item We first suppose that $Y_0 = x_0$ where $ x_0 \in \R$. In this case
its law  $z^0$ equals   $ \delta_{x_0}$ i.e. Dirac measure in $x_0$.
In  Exercise 7.3.3 of \cite{SV79},  the following Krylov type estimate is 
provided:
$$ \left \vert E  \left( \int_0^T f(t, Y_t) dt \right ) \right \vert
\le {\rm const}
\Vert f \Vert_ {L^2 ([0,T] \times \R)},$$
for every smooth function  $f: [0,T] \times \R
\rightarrow \R$ with compact support. 
This  implies the existence of 
a density $ (t,y) \mapsto p_t (x_0, y)$ for the 
measure $(t,y) \mapsto E(\int_0^T f(t,Y_t)dt)$.
and
$$\left \vert  \int_{[0,T] \times \R} f(t, y) p_t(x_0,y) dt dy 
 \right \vert \le
 {\rm const}
\Vert f \Vert_ {L^2 ([0,T] \times \R)},
$$
and  ${\rm const}$ does not depend on $x_0$, but only on
lower and upper bounds of $a$. 
This obviously implies that 
$$ \sup_{x_0 \in \R} \int_{[0,T] \times \R}   p_t^2 (x_0, y) dt dy < \infty. $$
This implies assertion 2.,  when $Y_0$ is deterministic.

If the initial condition $Y_0$ is any law $z^0 (dx)$, then 
clearly the density of $Y_t$ is
$z_t(dy) = \rho(t,y) dy $
where $\rho(t,y)  = \int_\R u_0 (dx) p_t (x, y) $.

Consequently, by Jensen's  inequality and Fubini's Theorem, 
\begin{equation*}
 \int_{ [0,T] \times \R  } \rho^2 (t, y) dt dy 
 \le  \int_\R  u_0 (dx) \int_{ [0,T] \times \R}  p_t^2 (t, x, y) dt dy 
< \infty.
\end{equation*}
\item The final assertion follows by 2. from Theorem \ref{P3.5}.
\end{enumerate}
\end{prooff}

Now we come back to the probabilistic representation of 
 equation (\ref{E1.1}).

Let us consider the solution $ u  \in  (L^1 \bigcap L^\infty) ([0,T] 
\times \R)$ from Proposition \ref{P3.1}, that is, $u$ solves
 equation (\ref{E3.1a}), in the sense of  (\ref{E3.1}),
assuming the
 initial condition $u_0$ is an a.e. bounded probability density.
Define 
\begin{equation} \label{E4.0}
\chi_u (t,x) := \left \{
\begin{array}{ccc}
\sqrt \frac{\eta_u (t,x)}{u(t,x)} \ & {\rm if} & \ u(t,x) \neq 0 \\
c_1  \ & {\rm if} & \ u(t,x) = 0,
\end{array} \right.
\end{equation}
where $c_1 \in \Phi(0)$.
Note that, because $\beta$ is non-degenerate and 
$\chi_u(t,x) \in \Phi(u(t,x))$
  $dt \otimes dx$-a.e.,  we have    $\chi_u \ge c_0  > 0, \ dt \otimes dx$-a.e.
Since $\chi_u$ is only defined 
$dt \otimes dx$-a.e, let us fix a Borel version.
According to Remark \ref{rsv}, it is possible to construct a 
(unique in law) process  
 $Y$ which  is the weak solution of
\begin{equation} \label{E4.0bis}
 Y_t = Y_0 + \int_0^t \chi_u(s,Y_s) dW_s
\end{equation}
where $W$ is a classical Brownian motion 
on some filtered probability space and $Y_0$ 
is a random variable so that  $u_0$ is the density of its law.

Consider now the law $v(t,\cdot)$ of the  process $Y_t$.
We set $a(t,x) = \frac{\chi_u^2(t,x)}{2}$.
 Since $a \ge c  > 0 $, 
 Proposition \ref{Psqintd} implies that  $v \in L^2([0,T] \times \R)$ 
and it solves the equation 
\begin{equation}\label{E4.1}
\left \{
\begin{array}{ccc}
\partial_t v &=&\partial_{xx}^2
(a v)  \\
v(0,x)& =& u_0(x).  
\end{array}
\right.
\end{equation}

On the other hand $u$ itself, which is a solution to 
(\ref{E3.1a}) (in the sense of (\ref{E3.1})),  is another
 solution of equation (\ref{E4.1}).
So, being in $(L^1 \bigcap L^\infty) ([0,T] \times \R)$, $u$  is
 also square integrable.
Setting $z_1 = v, z_2 = u $,  Theorem \ref{P3.5}  implies that
$v = u, \ dt\otimes dx $-a.e. 

 Since $u \in C([0,T], L^1(\R)$ and $Y$ has continuous sample paths, it follows that
 $u(t,\cdot) = v(t, \cdot), \ dx$-a.e. for all $t \in [0,T]$. 

The considerations above  prove the existence part
of the following representation theorem,
at least in the non-degenerate case.

\begin{theo} \label{T4.2} 
Suppose that  Assumption \ref{H3.0} holds.
Let  $u_0 \in L^1 \bigcap L^\infty$ such that $u_0 \ge 0$ and
$\int_\R u_0(x)dx = 1$.
Suppose the multi-valued map $\Phi$  is  bounded and non-degenerate.
Then there is a process $Y$, unique  in law, such that there
exists $ \chi \in  (L^1 \bigcap L^\infty) ([0,T] \times \R)$ with
\begin{equation}
\label{E4.2}
\left \{
\begin{array}{ccc}
Y_t &=& Y_0 + \int_0^t \chi(s,Y_s) dW_s \ {\rm (weakly)} \\
\chi(t, x) &\in & \Phi(u(t,x)), \ {\rm for} \ dt \otimes dx {\rm -a.e.} \
(t,x) \in [0,T] \times \R \\
{\rm Law \quad density \quad of }  \quad Y_t &=& u(t,\cdot) \\
u(0,\cdot) &=& u_0,     
\end{array}
\right.
\end{equation}
with $u \in C([0,T]; L^1(\R)) \bigcap
 L^\infty( [0,T] \times\R) $.
\end{theo}
  \begin{rem} \label{R4.2}
If $\Phi$ is single-valued then $\chi_u \equiv \Phi(u)$.
\end{rem}

\begin{proof}

Existence has been established above.
Concerning uniqueness, given two solutions 
 $Y^i, i = 1, 2 $ of (\ref{E4.2}) i.e.  (\ref{E1.2}). 
By $u_i(t, \cdot),  i = 1, 2 $,  we denote  the law densities of 
respectively $Y^i, i =1,2$ with corresponding $\chi_1$ and $\chi_2$.

The multi-valued version of Theorem \ref{TI.1}
says that $u_1$ and $u_2$ solve equation (\ref{E1.1}) in the sense 
of distributions,
so that by Proposition \ref{P3.1} (uniqueness for (\ref{E3.1})) 
 we have $u_1 = u_2$, and also $\chi_1 = \chi_2$ a.e.

We note that, since $Y^i_t$ has a law density for all $t > 0$, the stochastic 
integrals in (\ref{E4.2}) are independent
 of the chosen Borel version of $\chi$.
Remark \ref{rsv} now implies that the laws of $Y^1$ and $Y^2$
(on path space) coincide.
\end{proof}

\begin{cor} \label{C4.2} 
Consider the situation of Theorem \ref{T4.2} and let
    $v_0 \in L^1 \bigcap L^\infty$ be such that $v_0 \ge 0$.
The unique solution $v$ to 
equation (\ref{E3.1a}) with initial condition $v_0$
is non negative for any $t \ge 0$.
Moreover, the mass $\int_\R v(t, x) dx$ does not depend on $t$. 
\end{cor}
\begin{proof} \
Set $\mu_0 = \int_\R v_0 (y) dy$, which we can suppose  to be
greater than $0$.
Then the function $u(t,x) = \frac{v(t,x)}{\mu_0} $
solves equation (\ref{E3.1a}) 
\begin{equation} \label{E3.1b}
\left \{
\begin{array}{ccc}
\partial_t u&=& \frac{1}{2}\partial_{xx}^2(\frac{\beta(\mu_0 u)}{\mu_0}), \\
u(0, \cdot )& =& \frac{v_0}{\mu_0}.  
\end{array}
\right.
\end{equation}
Hence, the result follows from Theorem \ref{T4.2}.
\end{proof} 
\begin{rem} \label{R4.6} 
We note that if $\Phi$ is merely bounded below by a strictly positive
constant on every compact set and if the solutions 
$u$ are continuous
on $[0,T] \times
\R$, then Theorem \ref{T4.2} and Corollary \ref{C4.2} still hold.
In fact, Stroock-Varadhan arguments contained in Remark
\ref{rsv} are still valid if $\chi_u$ is strictly positive
on each compact set.
\end{rem}

\subsection{The degenerate case} \label{sub4.2}

The degenerate  case is much more difficult
and will be analyzed in detail in  the forthcoming paper \cite{BRR}. 
In this subsection we only explain the first two steps in the special
case where our $\beta$ of Section 1 is of the form
$\beta(u) = \Phi^2(u) u$ and the following properties hold:
\begin{property}\label{E4.6prime}
$\Phi: \R \rightarrow \R$  is  single-valued,  continuous on 
$\R-\{0\}$.
\end{property}
\begin{rem}\label{R4.6prime}
A priori $\Phi(0)$ is an interval; however, by convention,
in this subsection, we will set $\Phi(0) := \liminf_{\varepsilon 
\rightarrow 0+} \Phi(u)$.
This implies that $\Phi$ is always lower semicontinuous.
\end{rem}
We furthermore assume that the initial condition $u_0$ and $\Phi$ are
such that we have for the corresponding solution $u$ to (\ref{E1.1})
(in the sense of Proposition \ref{P3.1})  the following:
\begin{property} \label{E4.7prime}
$\Phi^2(u(t,\cdot)): \R \rightarrow \R$   is  Lebesgue  almost 
  everywhere  continuous for $ dt $ a.e $t \in [0,T]$.
\end{property}

\begin{rem} \label{R4.7}
As will be shown in \cite{BRR} Property \ref{E4.7prime}
is fulfilled in many interesting cases, for a large class of intial
conditions.
In fact, we expect to be able to show that $u(t,\cdot)$ is even 
 locally of
bounded
variation if so is $u_0$.
\end{rem}

\begin{prop} \label{PDeg} Suppose that Property \ref{E4.6prime} holds. 
Let $u_0 \ge 0$ be a bounded integrable real function such that $\int_\R
u_0(x) dx = 1$ and the corresponding solution $u$ to (\ref{E1.1}) 
satisfies  Property \ref{E4.7prime}.
Then, there is at least one  process $Y$ such that
\begin{equation}
\label{EProbDeg}
\left \{
\begin{array}{ccc}
Y_t &=& Y_0 + \int_0^t \Phi(u(s,Y_s)) dW_s  \ {\rm in \ law} \\
{\rm Law \quad density } (Y_t) &=& u(t,\cdot), \\
u(0,\cdot) &=& u_0
\end{array} \right.
\end{equation}
\end{prop}
\begin{cor} \label{CNondeg}
Suppose that  Property \ref{E4.6prime} holds.
Let  $u_0 \in L^1 \bigcap L^\infty$ be such that $u_0 \ge 0$ and
that the corresponding solution $u$ to (\ref{E1.1}) satisfies 
 Property \ref{E4.7prime}. 
The unique solution $u$ to 
equation (\ref{E3.1a}) is non-negative for any $t \ge 0$.
Moreover, the mass $\int_\R u(t, x) dx$ is constant in $t \in [0,T]$. 
\end{cor}

\begin{prooff} \ (of Proposition \ref{PDeg}).  \
We denote   the solution  to 
equation (\ref{E3.1}), by $u = u(t,x)$. 

Let $ \varepsilon \in ]0,1]$
and set $\beta_\varepsilon  (u) = (\Phi(u) +
\varepsilon)^2 u, \quad  \Phi_\varepsilon(u) = \Phi(u) + \varepsilon,
u \in \R $.
Proposition \ref{P3.1} provides the solution $u = u^\varepsilon$ to the 
deterministic PDE equation (\ref{E3.1})
$$
\left \{
\begin{array}{ccc}
\partial_t u&=& \frac{1}{2} \partial_{xx}^2(\beta_\varepsilon  (u)), \\
u(0,x)& =& u_0(x).  
\end{array}
\right.
$$

 We consider the
unique solution $Y = Y^\varepsilon$ in law of 
\begin{equation}
\label{EProbDeg1}
\left \{
\begin{array}{ccc}
Y_t &=& Y_0 + \int_0^t \Phi_\varepsilon(u(s,Y_s)) dW_s  \\
{\rm Law \quad density } (Y_t) &=& u^\varepsilon(t,\cdot) \\
u^\varepsilon (0,\cdot) &=& u_0.
\end{array} \right.
\end{equation}

Since $\Phi + \varepsilon$ is non-degenerate,
this is possible because of Theorem \ref{T4.2}.

Since $\Phi$ is bounded, using Burkholder-Davies-Gundy
 inequality 
one obtains 
\begin{equation} \label{E4.7''}
 \E \{Y^\varepsilon_t - Y^\varepsilon_s \}^4 \le {\rm const}
(t-s)^2, \forall \varepsilon  > 0.
\end{equation}
where const does not depend on $\varepsilon$.
 Using the Garsia-Rodemich-Rumsey lemma (see for instance 
\cite{BY82}, (3.b), p. 203), we obtain that
$$ \sup_{\varepsilon > 0} E \left( \sup_{s, t \in [0,T]}
\frac{ \vert Y^\varepsilon_t - Y^\varepsilon_s
 \vert^4  }{ \vert t - s \vert} \right)  <
\infty. $$
Consequently, using Chebyshev's inequality 
$$ \lim_{\delta \rightarrow 0} \sup_{\varepsilon   > 0}
P (\{ \sup_{s, t \in [0,T], \vert t- s\vert \le \delta }
 \vert Y^\varepsilon_t - Y^\varepsilon_s
 \vert  > \lambda \} )  = 0, \ \forall \lambda  > 0. $$
This implies condition (4.7) of Theorem 4.10 in Section 2.4 
of \cite{ks}.
Condition (4.6) of the same theorem requires
$$ \lim_{\lambda \rightarrow + \infty} 
\sup_{\varepsilon > 0} P \{ \vert Y_0^\varepsilon \vert \ge \lambda \} =0.$$
This is here trivially satisfied since the law of $ Y_0^\varepsilon$
is the same for all $\varepsilon$.
Thus the same theorem implies 
that the family of laws of $Y^\varepsilon, \varepsilon  > 0$,   is tight.

Consequently, there is a subsequence $Y^n := Y^{\varepsilon_n}$
 converging in law
(as $C[0,T]$-valued random elements)  to some process $ Y$.
We set $\Phi_n := \Phi_{\varepsilon_n}$ and 
 $u^n := u^{\varepsilon_n}$ where we recall that
 $u^n(t, \cdot) $ is  the law of $Y^n_t$.

We also set $X^n_t = Y^n_t - Y^n_0$.
Since 
$$ [X^n]_t = \int_0^t \Phi_n^2 (u^n(s, Y^n_s)) ds, $$
and $E([X^n]_T)$ is finite, $\Phi$ being bounded, 
the continuous local martingales $X^n$ are indeed  martingales. 

By  Skorokhod's theorem  there is a new  probability space
$(\tilde \Omega, \tilde \shf, \tilde  P)$ 
and processes $\tilde Y^n$, with the same distribution
as $Y^n$ so that $\tilde Y^n$ converge $\tilde P$-a.e. to some process
 $\tilde Y$, of course distributed as $Y$,  as $C([0,T])$- random elements.
In particular, the processes $\tilde X^n := \tilde Y^n - \tilde Y^n_0$
 remain martingales
with respect to the filtrations generated by themselves.
We denote   the sequence   $\tilde Y^n$ 
(resp. $ \tilde Y$), again by $Y^n$ (resp. $Y$).

\begin{rem}\label{RUY}
We observe that, for each $t \in [0,T]$,
$u(t, \cdot)$ is the law density of $Y_t$.
Indeed, for any $t \in [0,T]$, $Y^n_t$ converges in probability
to $Y_t$; on the other hand $u^n(t, \cdot)$, which is the law of $Y^n_t$,
 converges to $u(t, \cdot)$  in $L^1(\R)$ uniformly in $t$,
 cf. \cite{BeC81}, Theorem 3 and the preceeding remarks.
\end{rem}

\begin{rem} \label{Rem4.10}
Let $\shy^n$ (resp.  $\shy$)
be the canonical filtration associated with $Y^n$ (resp. $Y$).

 We set 
$$ W^n_t = \int_0^t \frac{1}{\Phi_n (u^n (s, Y^n_s))} dY^n_s.$$
Those processes $W^n$ are standard $(\shy^n_t$) -Wiener processes 
since $[W^n]_t =
t$ and because of L\'evy's characterization theorem of Brownian motion.
Then one has
$$  Y^n_t = Y_0^n + \int_0^t \Phi_n(u^n(s, Y^n_s)) dW^n_s.$$
\end{rem}

We aim to prove first that 
\begin{equation}\label{EFinDeg}
 Y_t = Y_0 + \int_0^t  \Phi(u(s, Y_s)) dW_s.
\end{equation}
Once the previous equation is established for the given $u$,
the statement of Proposition \ref{PDeg} would be completely
proven because of Remark \ref{RUY}.
In fact, that Remark shows  in particular  the third line of (\ref{EProbDeg}).

We consider the stochastic process  $X$ (vanishing at zero)
defined by $X_t = Y_t - Y_0$. We also set again $X^n_t = Y^n_t - Y_0^n$.

Taking into account Theorem 4.2 of Ch. 3 of \cite{ks}, 
as in Remark \ref{RMistake1},
 to establish (\ref{EFinDeg}) it will be enough to prove that $X$ is a $\shy$-
martingale with quadratic variation
$[X]_t = \int_0^t \Phi^2(u(s, Y_s)) ds. $

Let $s, t \in [0,T]$ with  $t > s$  and 
  $\Theta$  a bounded continuous function from
$C([0,s]) $ to $\R$.

 In order to prove the martingale property for $X$, we need to show that
$$ E\left((X_t - X_s) \Theta(Y_r, r \le s) \right) = 0.$$
But this follows because $Y^n \rightarrow Y$ a.s.  (so  $X^n
\rightarrow X$
a.s.) as  $C([0,T])$-valued processes; 
so for each $t \ge 0, \ X^n_t \rightarrow
X_t$ in $L^1(\Omega)$ since $(X^n_t, n \in \N)$ is bounded in 
$L^2( \Omega) $ and
  $$ E \left((X^n_t - X^n_s) \Theta(Y^n_r, r \le s) \right) = 0.$$
It remains to show that $X_t^2 - \int_0^t \Phi^2 (u(s,Y_s)) ds,
  t \in [0,T]$,
 defines a $\shy$-martingale, that is,
 we need to verify that 
$$  E\left((X_t^2 - X_s^2 - \int_s^t \Phi^2(u(r, Y_r)) dr )
\Theta(Y_r, r \le s) \right) = 0.$$
The left-hand side decomposes into $2(I^1(n) + I^2(n) + I^3(n)) $ 
where 
\begin{eqnarray*}
 I^1(n) &=&  E \left(  (X_t^2 - X_s^2 - \int_s^t \Phi^2 (u(r, Y_r) )dr)  \Theta(Y_r,
r \le s) \right ) \\ & - &  E \left ( \left((X^n_t)^2 - (X^n_s)^2 - \int_s^t
\Phi^2(u(r, Y^n_r)) dr \right) 
 \Theta(Y^n_r, r \le s) \right), \\
 I^2(n) &=&  E\left(\left( (X^n_t)^2 - (X^n_s)^2 - \int_s^t   \Phi_n^2(u^n (r, Y^n_r)) dr
  \right) \Theta(Y^n_r, r \le s)
\right ), 
\end{eqnarray*}
and 
$$
 I^3(n) =  E \left( \int_s^t \left(  \Phi^2_n(u^n (r, Y^n_r))
-   \Phi^2(u (r, Y^n_r))   \right) dr \Theta(Y^n_r, r \le s)
\right ). $$

We start by showing the convergence of 
$I^3(n)$.
 Now $\Theta(Y^n_r, r \le s)$ 
 is dominated by a constant.
Therefore, since $\Phi_n, \Phi$ are uniformly bounded and
$ a^2 - b^2 = (a-b)(a+b)$, by the Cauchy-Schwarz inequality, it
suffices to consider the expectation of 
\begin{equation} \label{EE4.9}  \int_s^t \left(  \Phi_n(u^n (r, Y^n_r)) 
-  \Phi (u (r, Y^n_r))  \right)^2  dr  
\end{equation}
which is equal to
\begin{eqnarray*}
\int_s^t E ( \Phi_n (u^n (r, Y^n_r)) &- & 
 \Phi(u (r, Y^n_r)) )^2 dr   \\
& = &  \int_s^t dr \int_\R  \left  (\Phi_n(u^n (r, y)) -
\Phi(u (r, y)) \right )^2
u^n (r,y) dy. 
\end{eqnarray*}
This equals
$J_1(n)  + J_2(n) - 2 J_3(n)$
where
\begin{eqnarray*}
J_1(n) &=&  \int_s^t  dr \int_\R   \Phi_n^2(u^n (r, y))u^n (r,y) dy \\
J_2(n) &=&  \int_s^t dr \int_\R  \Phi^2(u (r, y))u^n (r,y)  dy \\
J_3(n) &=&  \int_s^t dr \int_\R   \Phi_n(u^n (r, y))\Phi(u (r, y)) 
  u^n (r,y)  dy.
\end{eqnarray*}
Define
$$
J := \int_s^t \int_\R   \Phi^2(u(r, y))u (r,y) dy = 
 \int_s^t \int_\R  \beta(u(r,y)) dy
$$
To show that $I^3(n) \rightarrow 0$ as $n \rightarrow \infty$,  it
suffices to show that 
\begin{equation} \label{E4.9prime}
\lim_{n \rightarrow \infty} J_1(n) = \lim_{n \rightarrow \infty}
J_2(n) = J
\end{equation}
and 
\begin{equation} \label{E4.10prime}
\liminf_{n \rightarrow \infty} J_3(n) \ge J.
\end{equation}

Now repeating exactly the same arguments as in Point 3.
of Proposition \ref{P3.1}, it follows that
$\Phi_n^2(u_n) u_n \rightarrow \Phi^2(u) u$ in  $\sigma(L^1, L^\infty)$
as $n \rightarrow \infty$ which immediately implies (\ref{E4.9prime}).

Furthermore, by Fatou's lemma and since $\Phi_n \ge \Phi$,
$$ \liminf_{n \rightarrow \infty} J_3(n)   \ge 
\int_0^t \int_\R  \liminf_{n \rightarrow \infty} \Phi(u^n(r,y))
\Phi(u(r,y)) u(r,y) dy dr$$
which by the lower semicontinuity of $\Phi$, implies
(\ref{E4.10prime}).

Now we go on with the analysis of $I^2(n)$ and $I^1(n)$. 
$I^2(n)$  equals  zero because
$X^n$  is a martingale with quadratic variation given by
$[X^n]_t = \int_0^t \Phi_n^2(u^n (r, Y^n_r)) dr$.

We treat finally $I^1(n)$. We recall that $X^n \rightarrow
X$ a. s. as a random element in $C([0,T])$ and that
the sequence $E\left
((X^n_t)^4\right ) $ is 
bounded, so $(X^n_t)^2$   are uniformly integrable. Therefore,
 we have 
$$ E \left( (X^n_t)^2 - (X^n_s)^2) \Theta(Y^n_r, r \le s) \right ) -
E \left((X_t^2 - X_s^2) \Theta(Y_r, r \le s) \right) \rightarrow 0,$$
when $n \rightarrow \infty$.
It remains to prove that
\begin{equation} \label{EDec}
    \int_s^t E\left(  \Phi^2(u (r, Y_r)) - 
\Phi^2(u (r, Y^n_r))   \Theta(Y^n_r, r \le s)  dr  \right )
 \rightarrow 0.
\end{equation}
Now, for fixed $dr$-a.e.  $r \in [0,T]$, 
 $\Phi(u(r, \cdot))$ has a Lebesgue zero set of discontinuities.
Moreover,  the law of $Y_r$  has a density. 
 So, let $N(r)$ be the null event of all $\omega \in \Omega$  such that 
$Y_r (\omega)$ is a point of discontinuity of $\Phi(u(r, \cdot))$.
For $\omega \notin N(r)$ we have 
$$ \lim_{n \rightarrow \infty} \Phi^2(u(r, Y_r^n (\omega))) =  \Phi^2(u(r, Y_r (\omega))).$$
Hence Lebesgue dominated convergence theorem implies (\ref{EDec}).

\medskip
\end{prooff}


{\bf ACKNOWLEDGEMENTS} 

\noindent The three authors are grateful to the three
Referees who gave significant suggestions
to improve the quality of the paper.\\
\noindent The second and third named authors would like
to thank Prof. Viorel Barbu for stimulating discussions.
Financial support through the SFB 701 at Bielefeld University and
NSF-Grant 0606615
 is gratefully acknowledged.


 \end{document}